\documentclass[11pt,twoside,reqno]{amsart}

\usepackage{amssymb}

\def\const{\text{\rm const}}

\def\sm{\setminus}
\def\ti{\tilde}

\def\PW{\text{\rm PW}}
\def\dist{\text{\rm dist}}
\def\supp{\text{\rm supp}\,}
\def\to{\rightarrow}

\def\bs{\bigskip}

\def\no{\noindent}

\def\R{{\mathbb R}}

\def\Z{{\mathbb{Z}}}
\def\C{{\mathbb{C}}}

\def\DD{{\mathcal{D}}}

\def\MM{{\mathcal M}}

\def\GG{{{{\mathbf{G}}}}}
\def\EE{{\mathcal E}}

\def\e{\varepsilon}

\def\L{\Lambda}
\def\l{\lambda}
\def\G{\Gamma}
\def\lan{\lambda_n}
\def\gan{\gamma_n}

\def\const{\text{\rm const}}

\def\dist{\text{\rm dist}}
\def\supp{\text{\rm supp}\,}

\def\PWa{\mathcal{P}\mathcal{W}_a}

\theoremstyle{plain}
\newtheorem{lemma}{Lemma}
\newtheorem{theorem}{Theorem}
\newtheorem{corollary}{Corollary}

\newtheorem{claim}{Claim}

\newtheorem{definition}{Definition}

\numberwithin{equation}{section}

\author{A.~Poltoratski}
\address{Texas A\&M University
\\ Department of Mathematics\\
College Station, TX 77843, USA\\
and \\
Department of Mathematics and Mechanics\\ St. Petersburg State University\\ St. Petersburg, Russia }
\email{alexeip@math.tamu.edu}
\thanks{Initial  work on this article (Sections 1-5) was supported by
NSF Grant DMS-1665264. The article was completed (Sections 6-9) with support of RNF grant 14-41-00010}

\title{Type alternative for Frostman measures}

\begin{document}

\begin{abstract} For a finite positive Borel measure $\mu$ on $\R$ its exponential type, $T_\mu$, is
defined as the infimum of  $a>0$ such that finite linear combinations of complex exponentials with
frequencies between 0 and $a$ are dense in $L^2(\mu)$. The definition can be easily extended from finite to 
broader classes of measures. In this paper we prove a new
formula for $T_\mu$ and use it to study growth and additivity properties of measures with finite positive type.
As one of the applications, we show that Frostman measures on $\R$ may only have type zero or infinity.
\end{abstract}

\maketitle

\section{Introduction}\label{sInt}

For $a>0$ denote by $\EE_a$ the family of complex exponential functions on $\R$ with frequencies between $0$ and $a$:
$$\EE_a=\{e^{isx}|\ s\in [0,a]\}.$$
If $\mu$ is a positive finite measure on $\R$ its exponential type $T_\mu$ is defined as
$$T_\mu=\inf\{a|\ \EE_a \textrm{ is complete in }L^2(\mu)\},$$
if the set on the right-hand side is non-empty and as infinity otherwise.
Recall that a family of vectors in a Banach space is complete in the space if finite linear combinations of
vectors from the family are dense in the space. The type problem, the problem of finding $T_\mu$ in terms of $\mu$,
has many connections in Fourier analysis, spectral analysis of differential operators and related fields, see for instance \cite{BS, Krein1, Krein2, Krein3, Khabibullin, Koosis, Type, CBMS}.
A formula for $T_\mu$ was recently obtained in \cite{Type} and further developed in \cite{CBMS}. One of the implications
of these results is that $L^2$ in the above definition of $T_\mu$ may be replaced with any $L^p,\ p>1,$ without changing the value
of $T_\mu$, i.e., that the $p$-type of $\mu$ for any $p>1$ is equal to the $2$-type. 

Although the type formula from \cite{Type} has improved  existing results and gave new examples in the area of the type problem, some
of the natural questions remained open. This note originated from one of such questions. Consider  the Poisson measure $\Pi$ on the real line, $$d\Pi(x)=\frac{dx}{1+x^2}.$$
Its type $T_\Pi$ can be easily shown to be equal to infinity.  On the opposite end of the scale, if one considers any measure with 'long' gaps in its
support (see Section \ref{sGap}), the type is equal to zero, as follows from another classical result, Beurling's gap theorem. 
The simplest
examples of measures of zero type are measures with semi-bounded (or bouneded) support, for instance, a restriction of $\Pi$ to a half-line.
The sharpness of
Beurling's theorem was demonstrated by M. Benedicks in \cite{Benedicks} where further examples of absolutely continuous measures of zero and infinite type
were constructed. Examples from \cite{Benedicks} focused on restrictions of the Poisson measure to unions of unit intervals. 

To obtain an example of an absolutely continuous measure with finite positive type one needs to utilize further results on the type problem.
Define the density $f$ to be equal to $e^{|n|}/(1+n^2)$ on each interval $[n,n+e^{-|n|}),\ n\in\Z$, and to zero elsewhere. Then the type of
$\mu,\ d\mu(x)=f(x)dx$, is equal to $2\pi$, as follows from the results of Borichev and Sodin  \cite{BS} or the results of \cite{Type}.

Notice that in the last example the density is extremely unbounded, while in the examples with more regular  densities the type always
comes out to be zero or infinity. This pattern persists over all known examples giving raise to the following natural question:
Can an absolutely continuous measure with bounded density have finite positive type?

In this note we give a negative answer to this question extending the result to a slightly wider class of measures. A positive measure $\mu$ on
$\R$ is Poisson-finite if
$$\int\frac{d\mu(x)}{1+x^2}<\infty.$$
The definition of type given above can be easily extended from finite to Poisson-finite and even wider classes of measures, see Section \ref{sPrelim}.

A positive measure $\mu$ on $\R$ is a Frostman measure if there exist positive constants $C$ and $\alpha$ such that for any interval $I\subset\R$,
\begin{equation}\mu(I)<C|I|^\alpha,
\label{eF}\end{equation}
where $|I|$ denotes the length of $I$. It is not difficult to see that every Frostman measure is Poisson-finite. Any absolutely continuous measure $\mu, \ d\mu(x)=f(x)dx$ with bounded density is a Frostman measure. Moreover, if
$f\in L^p(\R),\ 1<p< \infty$ then the measure is Frostman since
$$\mu(I)\leqslant ||f||_p|I|^{1/p}.$$

A measure is doubling if there exists a constant $C>0$ such that
$$\mu((x-2r,x+2r))<C\mu((x-r,x+r))$$
for any $x\in \R$ and any $r>0$. Under an additional restriction that $\mu(I)<D$ for all intervals $I$ of fixed (unit) length and some $D>0$,
a doubling measure is a Frostman measure.  Among singular measures, a standard Cantor measure on $[0,1]$ periodically extended to the rest of $\R$
is a relevant example of a  singular Frostman measure.

As we show in Theorem \ref{tFrost}, Section \ref{sTypeF}, the type of any Frostman measure can only be equal to zero or infinity. Our result relates local properties of a measure to its type, whereas previously known results only explored the relations between the type
and asymptotic properties of the measure near infinity.

To prove Theorem \ref{tFrost} we first  obtain a new version of the type formula in Section \ref{sTypeL}, Theorem \ref{tLp}. Unlike the previous versions, the formula
does not include a $\mu$-finite weight inherited from Bernstein's version of the type problem, see Section \ref{sTypeB}. Such an improvement
may be useful in applications, such as the problems discussed here.

If a measure with bounded  density cannot have a finite positive type, it is natural to ask  how fast the density
of a measure of such a type must grow. We give an answer to this question in Section \ref{sGrowth}, Theorem \ref{tGrowth}. In Section \ref{sAdd} we investigate
additivity properties of type.

\section{Preliminaries}\label{sPrelim}

 If $f$ is a function from $L^1(\R)$ we denote by $\hat f$ its Fourier transform

\begin{equation}\hat f(z)=\int_\R f(t)e^{- i z t}dt.\label{FTdef}\end{equation}

 Let $M$ be a set of all finite Borel complex measures on the real line. Similarly, for  $\mu\in M$ we define

$$\hat \mu(z)=\int_\R e^{- i z t}d\mu(t).$$

  Via Parseval's theorem, the Fourier transform may be extended to be a unitary operator from $L^2(\R)$ onto itself
and can be defined for even broader classes of distributions.

If one extends all functions from $L^2(-a,a)$ as $0$ to the rest of the line, one can apply the Fourier transform
to all such functions and obtain the Paley-Wiener space of entire functions
$$\PWa=\{\hat f|\ f\in L^2(-a,a)\}.$$


The definition of type $T_\mu$, given in the introduction for finite measures $\mu$, can be naturally extended to broader classes
of measures on $\R$:
$$T_\mu=2\inf\{a|\ \PWa\cap L^2(\mu) \textrm{ is dense in }L^2(\mu)\},$$
if the set of such $a$ is non-empty and as infinity otherwise. 
Via this definition one can consider the type problem in
the class of polynomially growing measures, i.e., measures $\nu$ such that $\nu=(1+|x|^n)\mu$ for some $n>0$ and some finite measure $\mu$.
Poisson-finite measures are polynomially growing measures with $n=2$.
We will denote the set of all polynomially growing Borel complex measures on the real line by $M_p$. The notation $M^+_p$ will be used
for the subset of positive measures.

Note that by duality, a family $F$ of functions is not dense in $L^q(\mu),\ q>1,$ iff there exists $$f\in L^p(\mu),\ \frac 1p +\frac 1q =1,$$ annihilating
the family, i.e., such that $$\int f\bar g d\mu=0$$ for every $g$ from the family. In such cases we write $f\perp F$.
We denote by $T^p_\mu$ the $p$-type of $\mu$ defined as
$$T^p_\mu=2\inf\{a|\ \exists f\in L^p(\mu),\ f\not\equiv 0, \ f\perp(\PWa\cap L^q(\mu)) \},$$
Note that $T_\mu=T^2_\mu$. Cases $p\neq 2$ were considered in several papers, see for instance articles by Koosis \cite{Koosis2} or Levin \cite{Levin2} for the case $p=\infty$. 

 It is well-known that if $\nu=(1+|x|^n)\mu$ then
$T^p_\nu=T^p_\mu$, which reduces the problem for polynomially growing measures to its original settings of finite measures.
The equality of $T^p_\mu,\ p>1$ to $T_\mu$ remains true for polynomially growing measures.

Suppose that $T^p_\mu>a$ for some finite measure $\mu, \ p>1$ and $a>0$. By our definitions this means that  there exists a function $$f\in L^q(\mu),\ \frac 1p + \frac 1q =1,$$  which annihilates $\EE_a$, i.e., such that
$$\hat f(s)=\int e^{-isx}f(x)dx=0\textrm{ for all }s\in [0,a].$$
An alternative way to extend the definition of type is to say that $T_\mu\geqslant a$ if $\widehat{f\mu}$, understood in the sense of distributions for
infinite measures, has no support on $[0,a]$ for some $f\in L^p(\mu),\ f\not\equiv 0$.

Hence the type problem becomes a version of the gap problem, which studies measures whose Fourier transform has a non-trivial gap in its support.
As we can see, to find $T^p_\mu$ is the same as to find what non-zero density $f\in L^q(\mu)$ gives the largest spectral gap for the
measure $f\mu$. It turns out that to approach the type problem it is beneficial to first solve the gap problem in the case $q=1$, which is no longer
a dual version of a $p$-type problem.

\section{Gap formula}\label{sGap}

We say that a polynomially growing measure $\mu\in M_p$ annihilates a Paley-Wiener space $\PWa$, and write $\mu\perp \PW_a$
 if for all functions $f\in \PW_a\cap L^1(|\mu|)$,
 $$\int f d\mu =0.$$
 Note that $\PWa$ contains a dense set of fast-decreasing functions belonging to $L^1(|\mu|)$ for any polynomially growing $\mu$.
 If $X$ is a closed subset of the real line we denote by $M_p(X)$ the set of polynomially growing measures supported on $X$.
 We denote by
 $\GG_X$ the
gap characteristic of $X$ defined as
\begin{equation} \GG_X=2\sup\{a\ |\ \exists\ \mu\in M_p(X) ,\ \mu\not\equiv 0,\text{ such that}\ \mu\perp\PWa\}.
\end{equation}
when the set is non-empty and zero otherwise.
Alternatively, one could define $\GG_X$ as the supremum of the size of the gap in the support of $\hat\mu$, taken over all non-zero finite
measures $\mu$ supported on $X$ (which explains the name).

The problem  of finding $\GG_X$ in terms of $X$ has many connections and applications, see for instance \cite{MiPo, CBMS} for results and further references. It was recently solved in \cite{GAP}. This version of
the gap problem is related to the version mentioned in the last section via the following observation.

\begin{lemma}[\cite{GAP, CBMS}]
For $\mu\in M^+_p$,
$$T^1_\mu=\GG_{\supp\mu}.$$
\end{lemma}

As we discussed in the last section, for $p>1$ the type $T^p_\mu=T_\mu$ is different from  $T^1_\mu$. Further formulas for the type will be discussed in the next two sections.

To give the formula for $\GG_X$  \cite{GAP} we will need the following definitions.

A sequence of disjoint intervals $\{I_n\}$ on the real line is called long (in the sense of Beurling and Malliavin)
if
\begin{equation}\sum_n\frac{|I_n|^2}{1+\dist^2(0,I_n)}=\infty\label{long}
\end{equation}
where $|I_n|$ stands for the length of $I_n$.
If the sum is finite we call $\{I_n\}$ short.

Let
$$...<a_{-2}<a_{-1}<a_0=0<a_1<a_2<...$$
\no be a two-sided sequence of real points. We say that the intervals $I_n=(a_n,a_{n+1}]$ form a short partition
of $\R$ if $|I_n|\to\infty$  as $ n\to \pm\infty$ and the sequence $\{I_n\}$ is short, i.e. the
sum in
\eqref{long} is finite.

Let $\Lambda=\{\lambda_1,...,\lambda_n\}$ be a finite set of points on $\R$. Consider the quantity

\begin{equation}E(\Lambda)=\sum_{\lambda_k,\lambda_j\in\L,\ \lambda_k\neq\lambda_j} \log|\lambda_k-\lambda_l|.\label{electrons}\end{equation}

As usual, a sequence of points $\L=\{\lan\}\subset\C$ is called discrete if it has no finite accumulation points.


\bs

Now we are ready to give the definition of $d$-uniform sequences used in our type formulas.

Let $\L=\{\lan\}$ be a discrete sequence of distinct real points and let $d$ be a positive number.
We say that  $\L$ is a $d$-uniform sequence if
there exists a short partition $\{I_n\}$ such that
\begin{equation}  \Delta_n=\#(\L\cap I_n)= d|I_n|+o(|I_n|)\   \text{as}   \ n\to\pm\infty \ \textrm{(density condition)}\label{density}\end{equation}
\no and
\begin{equation}  \sum_n \frac{\Delta_n^2\log|I_n|-E_n}{1+\dist^2(0,I_n)}<\infty\ \ \ \ \ \textrm{(energy/work condition)},\label{energy}\end{equation}
\no where

 $$E_n=E(\Lambda\cap I_n).$$

\bs

The following formula for the gap characteristic of a closed set was obtained in \cite{GAP}, see also \cite{CBMS}.

\begin{theorem}\label{MAINGAP}

$$\GG_X=2\pi\sup\{ d  |\ X\textrm{ contains a $d$-uniform sequence}\},$$

if the set on the right is non-empty and $\GG_X=0$ otherwise.

\end{theorem}

\section{Type formula in Bernstein's settings}\label{sTypeB}

We first approach the type problem in the settings of Bernstein's
weighted uniform approximation.

  Consider a weight $W$, i.e. a lower semicontinuous function $W:\R\to [1,\infty]$
that tends to $\infty$ as $x\to\pm\infty$. The space $C_W$ is the space of all continuous functions on $\R$ satisfying
$$\lim_{x\to\pm\infty} \frac{f(x)}{W(x)}=0.$$
We define the semi-norm in $C_W$ as
$$||f||_W=|| fW^{-1}||_\infty.$$
Finding conditions on $W$ ensuring completeness of polynomials or exponentials in $C_W$ is a classical problem, see for instance
\cite{Bernstein, Lub, Mergelyan, CBMS} for further discussion and references.

For a weight $W$ we define
$$\GG_W=\inf \{a|\ \EE_a\textrm{ is complete in }C_W\}.$$
We put $\GG_W=0$ if the last set is empty.

\begin{theorem}[\cite{CBMS}]\label{tBmain}
$$\GG_W=2\pi\sup \left\{d\ |\  \sum\frac{\log W(\lan)}{1+\lan^2}<\infty\textrm{ for some $d$-uniform sequence }\L \right\},$$
if the set is non-empty, and $0$ otherwise.
\end{theorem}

As was shown by A. Bakan \cite{Bakan}, $\EE_a$ is complete in  $L^p(\mu),\ 1\leqslant p\leqslant \infty,$ iff there exists a
weight $W\in L^p(\mu)$ such that $\EE_a$ is complete in $C_W$. This result yields the following corollary from the last theorem.
For $\mu\in M^+_p$, we call a weight $W$ a $\mu$-weight if $$\int Wd\mu<\infty.$$

\begin{corollary}[\cite{Type}]\label{Typemain}
Let $\mu$ be a finite positive measure on the line. Let $1<p\leqslant\infty$ and $d>0$ be  constants.

  Then $T^p_\mu\geqslant  2\pi d$
if and only if for any  $\mu$-weight $W$  there exists
a $d$-uniform sequence $\L=\{\lan\}\subset \supp \mu$ such that
\begin{equation}\sum\frac{\log W(\lan)}{1+\lan^2}<\infty.\label{ur4}\end{equation}

\end{corollary}

Note that this statement implies $T^p_\mu=T_\mu$ for all $p>1$, the property mentioned in previous sections. We will need these statements to obtain a new version of the type formula in the next section.

\section{Type formula in $L^p$-settings}\label{sTypeL}

\begin{definition}
If $\L=\{\lan\}\subset\R$ is a discrete sequence of distinct points we denote by $\L^*=\{\lan^*\}$ the sequence
of closed intervals such that each $\lan^*$ is centered at $\lan$ and has the length
equal to one-third of the distance from $\lan$ to the rest of $\L$. Note that
then the intervals $\lan^*$ are pairwise disjoint.

\end{definition}

\begin{definition}
If $\L$ is a discrete sequence we will write that $D_1(\L)=d$ if there exists
a short partition on which $\L$ satisfies \eqref{density}.

\end{definition}

Note that if $\L$ satisfies \eqref{density} with some  $d$ on some short partition then the asymptotic density of $\L$ is $d$ and therefore
$\L$ cannot satisfy \eqref{density} with any other $d$ on a different short partition, which implies correctness of the last definition.

\begin{theorem}\label{tLp} Let $\mu\in M^+_p$. Suppose that $T_\mu<\infty$. Then
$$T_\mu=2\pi\max\{ d|\ \exists\ d\textrm{-uniform }\L=\{\lan\}\textrm{ such that }\sum\frac{\log\mu(\lan^*)}{1+n^2}>-\infty\}$$
if the set of such $d$ is non-empty and $T_\mu=0$ otherwise.
\end{theorem}

Note that he maximum on the right-hand side exists whenever the set is non-empty. The maximal sequence can always be chosen inside the
support of the measure, which will be useful for us in applications.

The statement covers the case of finite type $T_\mu$, which is needed for Theorem \ref{tFrost} in the next section. We would like to leave it as an open problem whether the formula can be extended to the infinite case, i.e., if it is true that the equation (with supremum in
place of the maximum) holds when $T_\mu=\infty$.
It is not difficult to see that if the supremum of the set on the right is infinite then $T_\mu=\infty$. It therefore remains to check
the opposite implication.

\begin{proof} As follows from our discussion in Section \ref{sPrelim}, it is enough to prove the theorem for finite $\mu$.

Let $T_\mu=2\pi d<\infty$. One can show that then there exists  a $\mu$-weight $W$ such that \eqref{ur4} is satisfied for
some $d$-uniform sequence $\L$ but is not satisfied for any $d+\e$-uniform sequence.
Define
$$s=\max\{ d|\ \exists\ d\textrm{-uniform }\{\l_{n_k}\}\subset\L,\ \sum\frac{\log\mu(\l_{n_k}^*)}{1+{n_k}^2}>-\infty\}$$
(note that $\max$ can be used in the last formula instead of $\sup$).
Suppose that $s<d$. Let $\G$ be the maximal subsequence in the last formula.

Then
  the subsequence $\L_1=\{\l_{n_k}\}= \L\setminus \G$   satisfies $D_1(\L_1)=\e>0$ and  the corresponding sequence of intervals $\L^*_1=\{\l^*_{n_k}\}$ has the property that
  for any subsequence $ \L_2=\{\l_{n_{k_l}}\}\subset\L_1$, $D_1(\L_2)>0$,
$$\sum_l\frac{\log\mu(\l_{n_{k_l}}^*)}{1+{n_{k_l}}^2}=-\infty.
$$
 Indeed, if there existed $$\L_2\subset \L_1,\ D_1(\L_2)>0$$
for which the last sum were finite, then by Lemma \ref{lsubuniform}, $\G\cup\L_2$ would have an $(s+\e)$-uniform subsequence with finite sum,
which would contradict maximality of $\G$.

Now one can obtain a contradiction in the following way. Define a new weight $W_1$ to be equal to
$$\frac{\max(W(\l_{n_k}),1/\mu(\l^*_{n_k}))}{1+{n_k}^2}$$ on
each interval $\l^*_{n_k}$ from $\L^*_1$ and equal to $W$ elsewhere. Then $W_1$ is a $\mu$-weight and therefore \eqref{ur4} must be
satisfied with some $d$-uniform sequence $\Phi$. Notice that then $\Phi$
intersects the intervals from $\L_1^*$ only for a subsequence of density zero, i.e., there exists a subsequence $\Theta=\{\theta_n\}\subset \L_1$
such that $D_1(\Theta)=D_1(\L_1)=\e$ and
$$\Phi\setminus\cup_{\l_{n_k}\in\Theta}\ \l^*_{n_k}=\emptyset.$$
Since $\L_1$ is an $\e$-uniform sequence, $\Theta$ is an $\e$-uniform sequence.
By Lemma \ref{l1}, $\Phi\cup \Theta$ is a $(d+\e)$-uniform sequence on which the original weight $W$ satisfies \eqref{ur4}, which contradicts our choice of $W$. Hence, $s\geqslant d$.

In the opposite direction, suppose that there exists a $d$-uniform sequence $\L$ such that
$$\sum\frac{\log\mu(\lan^*)}{1+n^2}>-\infty.$$
Note that every $\mu$-weight $W$ satisfies
$$\min_{\lan^*} W(x)<\frac {C}{\mu(\lan^*)}$$
on every $\lan^*$, with $$C=\int Wd\mu.$$ The sequence of
points where the minima occur will give us a $d$-uniform sequence on which $W$ satisfies \eqref{ur4}. Thus
$T_\mu\geqslant 2\pi d$.
\end{proof}

\section{Type of Frostman measures}\label{sTypeF}

In this section we solve the problem of type for measures with bounded densities discussed in the introduction. As it turns out, our result
can be formulated for a broader class of Frostman measures, which seems
to be the right class for such a statement due to the elementary property
that $\log|I|^\alpha=\alpha\log|I|$.

Recall that a positive measure $\mu$ on $\R$.
is a Frostman measure if there exist positive constants $\alpha$ and $C$ such
that
\begin{equation}\mu((x-\epsilon,x+\epsilon))<C\epsilon^\alpha\label{eFrost}\end{equation}
for all $\epsilon>0, x\in \R$.
It easily follows that Frostman measures are Poisson-finite.

\begin{theorem}\label{tFrost}
If $\mu$ is a Frostman measure then $T_\mu$ equals either 0 or $\infty$.

\end{theorem}

\begin{proof} If $T_\mu=2\pi d,\ 0<d<\infty,$ then there exists a $d$-uniform sequence $\L$ such that $\L^*$ satisfies
\begin{equation}\sum\frac{\log\Delta_n}{1+n^2}>-\infty,\label{eTypeL}\end{equation}
where $\Delta_n=\mu(\lan^*)$.
 WLOG $\mu$ satisfies
\eqref{eFrost} with $C=1$ and some $\alpha,\ 0<\alpha\leqslant 1$. Divide the interval $\lan^*$ into $M_n$ equal subintervals,
so that
$$\frac 12 [\Delta_n/6]^{1/\alpha}\leqslant |\lan^*|/M_n \leqslant [\Delta_n/6]^{1/\alpha}.$$
Then the mass of each subinterval is at most $\Delta_n/6$ (and hence $M_n\geqslant 6$). Since the total mass is $\Delta_n$, there exist at least $3$ subintervals of mass at least
\begin{equation}\Delta_n/2M_n\asymp \Delta_n^{1+\frac 1\alpha}/|\lan^*|.\label{emass}\end{equation}
Let $l^n_1$ and $l^n_2$ be two of these intervals, not adjacent to each other. Then the distance $d_n$ between the intervals is greater or equal to
the length of each interval and its logarithm can be estimated
from below by
$$\log d_n\geqslant \log |\lan^*|/M_n\geqslant C_1\log\Delta_n +C_2,
$$
which together with \eqref{eTypeL}
implies 
\begin{equation}
\sum\frac{\log d_n}{1+n^2}>-\infty.\label{e00001}
\end{equation}
Consider the sequence $\G=\{\gamma_m\}$ composed of the centers of all such intervals $l^n_1, l^n_2$ for all $n$.
Since the sequence $\L$ was $d$-uniform and because of \eqref{e00001}, $\G$ is a $2d$-uniform sequence. Our construction  implies
that each of the intervals $\gamma^*_m$ contains one of the intervals  $l^n_1, l^n_2$, whose mass is at least the expression in \eqref{emass}.
Since $\Delta_n=\mu(\lan^*)$ satisfy \eqref{eTypeL}, the sequence of intervals $\{\gamma^*_m\}$ satisfies
$$\sum\frac{\log\mu(\gan^*)}{1+n^2}\gtrsim\sum\frac{\log\Delta_n}{1+n^2}+\const>-\infty$$
which implies $T_\mu\geqslant 4\pi d$ contradicting
our initial assumption.
\end{proof}

 As was discussed in the introduction, the last theorem has the following corollaries.

\begin{corollary}\label{c1}
Let $\mu=fm$, where $f\in L^p(\R),\ p>1$. Then $T_\mu$ is either 0 or infinity

\end{corollary}

\begin{corollary}
If $\mu$ is a doubling measure such that $$\sup_{x\in\R}\mu((x,x+1))<\infty$$ then $T_\mu$ is either 0 or infinity.

\end{corollary}

\section{Additivity properties of type}\label{sAdd}

It is natural to ask if the type of a measure satisfies any additivity conditions, in general or in special cases.
Our first observation in this direction is that the inequality $T_{\mu+\nu}\lesssim T_\mu +T_\nu$ fails in general.
Indeed, let $\mu$ and $\nu$ be the restrictions of Lebesgue measure $m$ 
to $\R_+$ and $\R_-$ respectively. Then
$T_\mu=T_\nu=0$, because if $f\in L^2(\nu)\ (\in L^2(\mu))$ then $\widehat{f\nu}\ (\widehat{f\mu})$ belongs to
the Hardy space $H^2\ (\bar H^2)$ and cannot vanish on a set of positive measure, unless $f\equiv 0$.
On the other hand $T_{\mu+\nu}=T_{m}=\infty$.

Further examples of this type, without semi-bounded supports, can be obtained using Beurling's theorem and choosing $\mu$ to be the restriction
of $m$ to the union of odd dyadic intervals $[2^{2n+1},2^{2n+2}],\ n\in \Z$ and $\nu$ as the restriction to
even intervals. Then once again $T_\mu=T_\nu=0$ but $T_{\mu+\nu}=T_{m}=\infty$.

In the opposite direction, we obviously have $T_{\mu+\nu}\geqslant \max(T_\mu,T_\nu)$. However, 
the inequality $T_{\mu+\nu}\geqslant T_\mu +T_\nu$ does not hold in general. To construct an example, consider
$$\mu=\sum_{n=-\infty}^{\infty} \delta_n\textrm{ and }\nu=\sum_{n=-\infty}^{\infty} \delta_{n+e^{-|n|}}.$$
Since both measures are supported by $1$-uniform sequences, $T_\mu=T_\nu=2\pi$ by Theorem \ref{tLp} (in this simple case the type can be calculated directly without any advanced results). Since the
support of $\mu+\nu$ does not contain any $d$-uniform sequences with $d>1$, $T_{\mu+\nu}=2\pi$ by Theorem \ref{MAINGAP} or \ref{tLp}.

Nonetheless, the following 'splitting' property holds for measures of finite positive type.

\begin{theorem} Let $\mu\in M^+_p$, $T_\mu=2\pi d<\infty$. Let $c_1,\ c_2$ be non-negative constants such that $c_1+c_2=d$.
	Then there exists a closed set $X\subset\R$ such that the measures $\mu_1=\mu |_X$, $\mu_2=\mu-\mu_1$ satisfy
	$$T_{\mu_1}=2\pi c_1, \ \ \ \ T_{\mu_2}=2\pi c_2.$$
	
\end{theorem}

Note that the analog of the above statement for measures of infinite type is false. If $\mu$ is a measure with bounded density such that
$T_\mu=\infty$, $c_1$ is finite positive and $c_2$ is infinite then any restriction of $\mu$ will again be a measure
with bounded density and therefore, by Theorem \ref{tFrost}, will not have its type equal to $2\pi c_1$.

\begin{proof}
By Theorem \ref{tLp} there exists a $d$-uniform sequence $\L$ such that the sequence of intervals $\L^*$ satisfies 
\begin{equation}\sum\frac{\log\mu(\lan^*)}{1+n^2}>-\infty.\label{esum}\end{equation}
If $\{I_n\}$ is a short partition corresponding to $\L$ from the definition of $d$-uniform sequences, choose 
a subsequence $\Gamma\subset\Lambda$ satisfying \eqref{density} on $\{I_n\}$ with $c_1$ in place of $d$. Note
that the energy condition for $\Gamma$ will then be satisfied on the same partition and therefore $\Gamma$ is 
a $c_1$-uniform sequence. 

Put $X=\cup_{\lan\in\Gamma} \lan^*$ and $\mu_1=\mu |_X$. We claim that $\mu_1$ and $\mu_2=\mu-\mu_1$ are the desired measures.
First, notice that $T_{\mu_1}=c_1$. Indeed, since $\Gamma$ is a $c_1$-uniform sequence, $T_{\mu_1}\geqslant 2\pi c_1$ by Theorem \ref{tLp}.  If 
$T_{\mu_1}=2\pi p>2\pi c_1$ then there exists a $p$-uniform sequence $\Phi$ satisfying \eqref{esum} for $\mu_1$. As was remarked after 
Theorem \ref{tLp}, we can choose $\Phi$ from $X$. Moreover, we can assume that each interval $\lan^*,\ \lan\in \Gamma$ contains
at least one point from $\Phi$. Then the intervals from $\Phi^*$ do not intersect the intervals $\lan^*,\ \lan\in \L\setminus\Gamma$.

Divide each interval $\lan^*,\ \lan\in \L\setminus\G$ into two equal subintervals. Note that at least one of them has mass of at least
one half of the original interval, with respect to $\mu$. For each $n$ choose the half-interval with the larger mass and denote the centers of these half-intervals by $\psi_n,\ \psi_n\in\lan^*$.
Note that $\Psi=\{\psi_n\}$ is a $c_2$-uniform sequence. By Lemma \ref{l1}, the sequence $\Phi\cup\Psi$ is a $(p+c_2)$-uniform sequence.
By our construction, the sequence of intervals $(\Phi\cup\Psi)^*$ satisfies \eqref{esum}, which implies that $T_\mu\geqslant 2\pi (p+c_2)>2\pi d$, a contradiction.

Similarly, $T_{\mu_2}=2\pi c_2$. 
\end{proof}

\section{Growth of density in the case of finite positive type}\label{sGrowth}

As we saw from Corollary \ref{c1}, a measure of finite positive type cannot have bounded density. It is natural to ask how
fast should its density grow. To this account we prove the following statement.

\begin{theorem}\label{tGrowth}
Let $\mu\in M^+_p$, be absolutely continuous, $\mu=fm$. Suppose that $T_\mu=2\pi d$, $0<d<\infty$.
For $x>0$ denote
$$\MM_f(x)= \textrm{ess sup }_{[-x,x]} f(x).$$
Then 
$$\int_{0}^{\infty}\frac{\log(1+\MM_f(x))}{1+x^2}dx=\infty.$$ 

\end{theorem}


\begin{proof} Suppose that
the integral from the statement is finite. 
 Consider the measure $\nu(x)=\mu(x)/(1+\MM_f(|x|))$.
This measure has bounded density and therefore has type 0 or infinity.
Note that for any  $\nu$-weight $W$, $U=W/(1+\MM_f) + 1$ is a $\mu$-weight.
Hence there exists a $d$-uniform sequence $\L$ such that $U$ satisfies \eqref{ur4}.
But
$$\log U(\lan)=\log W(\lan) - \log (1+\MM_f(|\lan|))$$
and 
$$\sum_n \frac{\log(1+\MM_f(|\lan|))}{1+\lan^2}\lesssim \int_0^\infty\frac{\log(1+\MM_f(|x|))}{1+x^2}<\infty.$$
Therefore $W$ satisfies \eqref{ur4} on the same sequence. 
Hence, by Corollary \ref{Typemain}, $T_\nu\geqslant 2\pi d$, which implies that $T_\nu=\infty$.

Therefore, for any $\nu$-weight $U$ there exists a $2d$-uniform sequence $\L$ which satisfies \eqref{ur4}.
Since any $\mu$-weight $W$ is equal to $U(1+\MM_f(|x|)$ for a $\nu$-weight $U$, similar to above we conclude
that $W$ satisfies \eqref{ur4} for a $2d$-uniform sequence $\L$. Thus by Corollary \ref{Typemain},  $T_\mu\geqslant 4\pi d$, a contradiction.
\end{proof}

Let us show that the condition of the last theorem is sharp in its scale. Let $M:\R_+\to [1,\infty)$ be an increasing function such that
$$\int_{0}^{\infty}\frac{\log M(x)}{1+x^2}dx=\infty.$$
Then there exists $\mu=fm$, $\MM_f(x)\leqslant M(x)$ such that $T_\mu$ is finite positive. Indeed, put $f$ to be equal
to $M(n)$ on every interval $$I_n=\left[n,n + \frac 1{M(n)}\right]$$ and zero elsewhere. One can show that
then every sequence $S\subset \Z$ such that
$$\sum_{n\in S}\frac{\log |I_n|}{1+n^2}>-\infty$$
has density zero. Hence, $\Z$ is the maximal $d$-uniform sequence from the statement of Theorem \ref{tLp} and 
$T_\mu=2\pi$.

\section{Additivity properties of uniform sequences}\label{sLem}

The following lemma was proved in \cite{GAP}.

\begin{lemma}\label{l-energy}
Let $\L$ be a sequence of real points and let $\{I_n\}$ be a short partition such that $\L$ satisfies
$$a|I_n|<\#(\L\cap I_n)$$
for all $n$ with some $a>0$ and the energy condition \eqref{energy} on $\{I_n\}$.
Then for any short partition $\{J_n\}$, there exists a subsequence $\G\subset\L$ which satisfies
$$\#(\L\sm\G)\cap J_n=o(|J_n|)$$
as $n\to\pm\infty$, and the energy condition \eqref{energy}
on $\{J_n\}$.

\end{lemma}

All $d$-uniform sequences satisfy the following simple properties.

\begin{lemma}\label{lsubuniform} Let $\L$ be a $d$-uniform sequence.

1) If $\Gamma\subset\L$ is a $c$-uniform sequence then $\L\setminus \Gamma$ contains a $(d-c)$-uniform subsequence.

2) if $\Gamma_1\subset\L$ is a $c_1$-uniform sequence and $\Gamma_2\subset\L$ is a $c_2$-uniform sequence, $\G_1\cap \G_2=\emptyset$,
then $\G_1\cup \G_2$ contains a $(c_1+c_2)$-uniform subsequence.

\end{lemma}

\begin{proof}
1) Let $\{I_n\}$ and $\{J_n\}$ be the short partitions from the definition of $d$-uniform sequences in Section \ref{sGap} for the sequences $\L$ and $\G$ correspondingly.
Let $L_n$ be a third short partition with the property that
$$\max \{|I_m|\ :\ I_m\cap L_n\neq \emptyset\} + \max \{|J_m|\ :\ J_m\cap L_n\neq \emptyset\}=
$$
\begin{equation}=o(|L_n|)\textrm{ as }n\to\infty.\label{e3part}
\end{equation}
By Lemma \ref{l-energy} there is a subsequence $\L'\subset\L$ such that
$$\#((\L\setminus\L')\cap L_n)=o(|L_n|)$$
and $\L'$ satisfies the energy condition on $\{L_n\}$. Notice that then $\L'\setminus\G$ satisfies
the energy condition and the density condition with the constant $d-c$ on $\{L_n\}$.

2) Similarly to the last part, if $\{I_n\}$ and $\{J_n\}$ are short partitions corresponding to $\G_1$ and $\G_2$, choose a short partition $\{L_n\}$ to
satisfy \eqref{e3part}. Choose a subsequence $\L'\subset \L$ like in the last part. Note that since $\L'$ satisfies
the energy condition on $L_n$, so does $\L'\cap (\G_1\cup \G_2)$. The last sequence also satisfies the density condition with the constant
$c_1+c_2$ on $L_n$.
\end{proof}

A union of uniform sequences is a uniform sequence, up to a sequence of density zero, if the original sequences
are separated from each other in the following precise sense.

\begin{lemma}\label{l1}
Let $\L=\{\lan\}$  be a $c$-uniform sequence and let $\G=\{\gan\}$ be a  $d$-uniform sequence. Let
$\{I_n\}$ be a sequence of disjoint intervals such that  $I_n$ is centered at $\lan$, $$|I_n|\leq\frac 13\dist (\lan,\L\setminus\{\lan\})$$ and
$$\sum\frac{\log|I_n|}{1+n^2}>-\infty.$$
Suppose that $\G\cap(\cup I_n)=\emptyset$. Then $\L\cup\G$ contains a $(c+d)$-uniform subsequence.
\end{lemma}

\begin{proof}
Denote by $p_n$ the intervals $$p_n=\frac 12 I_n.$$ Choose the intervals $q_n$ to be centered
at $\gan$ and such that
$$|q_n|=\frac 13 \dist (\gan, \cup I_n\cup\G\setminus\{\gan\}).$$
Note that then the intervals $q_n,p_n$ are  disjoint and satisfy
\begin{equation}\sum \frac{\log |p_n|}{1+n^2}>-\infty,\ \sum \frac{\log |q_n|}{1+n^2}>-\infty.\label{e100}\end{equation}

As follows from Lemma \ref{l-energy}, one can choose a short partition $\{S_n\}$ on which
both sequences $\L$ and $\G$ have subsequences which satisfy \eqref{energy} and \eqref{density} with constants $c$ and $d$
correspondingly. WLOG we will assume that the subsequences are equal to $\L$ and $\G$. We will also assume that
the endpoints of $S_n$ do not
fall into any of the intervals $p_n,q_n$ and that $$|S_n|=\frac 1c\#\L\cap S_n=\frac1d\#\G\cap S_n$$ (omitting the $o(|\cdot|)$ term in subsequent formulas).

For each $k\in\Z$ consider the functions $u_k$ and $v_k$ defined as follows. Each function $u_k$ is continuous and piecewise linear. It is zero on
$(-\infty,a_k)$, where $S_k=(a_k,b_k)$ and its derivative is zero outside of $(\cup p_n)\cap S_k$. On each $p_n, p_n\subset S_k$, $u_k$ grows linearly by one. It follows that $u_k=c|S_k|$ on $(b_k,\infty)$. Repeat the same construction for $v_k$ with the intervals $q_n$ in place of $p_n$.

Now define the functions $\phi_k$ as $u_k(x)-cx$ on $S_k$ and as 0 outside of $S_k$. Define $\psi_k$ as $v_k(x)-dx$ on $S_k$ and as 0 outside of $S_k$.
Employing some of the techniques used in \cite{GAP} we notice the following.

\begin{claim}
The functions $\phi_k$ and $\psi_k$ belong to the Dirichlet class $\DD$ with
$$||\phi_k||_{\DD}^2= \frac 1\pi (c^2|S_k|^2\log|S_k|-E_{\L\cap S_k})-\sum_{p_n\subset S_k} \log |p_n|+ O(|S_k|^2) ,$$
$$||\psi_k||_{\DD}^2= \frac 1\pi (d^2|S_k|^2\log|S_k|-E_{\G\cap S_k})-\sum_{q_n\subset S_k} \log |q_n|+ O(|S_k|^2) .$$
and
$$||\phi_k+\psi_k||_{\DD}^2= \frac 1\pi ((c+d)^2|S_k|^2\log|S_k|-E_{(\L\cup\G)\cap S_k})-
$$$$\sum_{p_n\subset S_k} \log |p_n|-
 \sum_{q_n\subset S_k} \log |q_n|+O(|S_k|^2) ,$$
as $k\to\infty$.
\end{claim}

\begin{proof}[Proof of Claim]
The function $\phi=\phi_k$ is a  function with bounded harmonic conjugate and compactly supported on $S_k=S$ derivative. Hence it belongs to $\DD$ and
its norm can be calculated as
$$||\phi ||_{\DD}^2=\int_\R \phi d\ti\phi= -\int_{\R}\ti \phi d\phi.$$
Integrating by parts,
$$-\int_{\R}\ti \phi d\phi=\frac 1\pi\int_{S}\phi'\left[\int_{S} \frac {\phi(t)dt}{t-x} \right] dx=$$
$$-\frac 1\pi\int_{S}\phi'\left[\int_{S} \log|t-x| \phi'(t)dt\right] dx.
$$
Let $p_1,p_2, ... p_n$ be the intervals from our construction above  inside $S$.
Recall that by our construction of $u=u_k$, $u'=1/|p_l|$ on each $p_l$ and to zero elsewhere. Since $\phi=u-cx$ we obtain
$$-\int_{S}\phi'\left[\int_{S} \log|t-x| \phi'(t)dt\right] dx=
-c^2\int_S\int_S \log|t-x| dtdx
$$$$
+2c^2\int_S\sum_{m=1}^n\frac 1{|p_m|}\int_{p_m}\log|t-x|dtdx
$$$$
-\sum_{m=1}^n\sum_{k=1}^m \frac 1{|p_m||p_k|}\int_{p_m}\int_{p_k}\log|t-x|dtdx=$$
$$
-I+II-III.$$
Recall that the points $\l_1,\l_2,...\lan$ are the centers of the intervals $p_1,p_2, ... p_n$. Because the distance between adjacent
intervals $p_l,p_{l+1}$ is at least $\max(|p_l|,|p_{l+1}|)$, for the last term we have
$$III= \sum_{1\leqslant m,k\leqslant n,\ m\neq k} \log|\l_m-\l_k| + \sum_{p_n\subset S_k}\log |p_n|+O(|S|^2).$$
Similarly,
$$II= 2c^2|S|^2\log |S| + O(|S|^2).$$
Finally, via elementary calculations,
$$I=c^2|S|^2\log|S|+O(|S|^2).$$
Combining the last three equations we obtain the statement for $\phi=\phi_k$. The equations for $\psi_k$ and  $\phi_k+\psi_k$ can be proved similarly.
\end{proof}

To finish the proof of the lemma, it remains to notice that the claim implies
$$\sum_k\frac{(c+d)^2|S_k|^2\log|S_k|-E_{(\L\cup\G)\cap S_k}}{1+k^2} +\const$$$$
\lesssim \sum_k ||\phi_k+\psi_k||_{\DD}^2+\const\lesssim \sum_k\left(
||\phi_k||_{\DD}^2+||\phi_k||_{\DD}^2\right)+\const
\lesssim $$$$\sum_k\frac{(c^2|S_k|^2\log|S_k|- E_{\L\cap S_k})+(d^2|S_k|^2\log|S_k|-E_{\G\cap S_k})}{1+k^2}+\const.$$
Since the sequences $\L$ and $\G$ are $c$- and $d$-uniform respectively on the partition $\{S_n\}$, the last sum is finite. Therefore
the sequence $\L\cup\G$  satisfies the energy condition \eqref{energy} on the partition $\{S_n\}$. Since it also satisfies  the density condition \eqref{energy} with $a=c+d$ on $\{S_n\}$, it is a $(c+d)$-uniform sequence.
\end{proof}

\end{document}